\documentclass{amsart}
\usepackage{setspace}
\usepackage{a4}
\usepackage{amssymb,amsmath,amsthm,latexsym}
\usepackage{amsfonts}
\usepackage{amsfonts}
\usepackage{graphicx}
\usepackage{textcomp}
\usepackage{cite}
\usepackage{enumerate}
\usepackage[mathscr]{euscript}
\usepackage{mathtools}
\newtheorem{theorem}{Theorem}[section]

\newtheorem{conjecture}[theorem]{Conjecture}

\newtheorem{proposition}[theorem]{Proposition}
\newtheorem{remark}[theorem]{Remark}

\title{This is the title}
\usepackage{hyperref}
\usepackage{enumerate}
\usepackage{mathtools}
\usepackage{amsmath}
\usepackage{tikz}
\usepackage{hyperref}
\usepackage{graphicx}
\usepackage{textcomp}
\usetikzlibrary{cd}
\usepackage{fancyhdr}

\begin{document}
\begin{center}
{\bf{   C*-ALGEBRAIC SMALE MEAN VALUE CONJECTURE AND DUBININ-SUGAWA DUAL MEAN VALUE CONJECTURE}\\
K. MAHESH KRISHNA}  \\
Post Doctoral Fellow \\
Statistics and Mathematics Unit\\
Indian Statistical Institute, Bangalore Centre\\
Karnataka 560 059 India\\
Email: kmaheshak@gmail.com\\

Date: \today
\end{center}

\hrule
\vspace{0.5cm}
\textbf{Abstract}: Based on Smale mean value conjecture \textit{[Bull. Amer. Math. Soc., 1981]} and Dubinin-Sugawa dual mean value conjecture \textit{[Proc. Japan Acad. Ser. A Math. Sci., 2009]} we formulate the following conjectures.\\
\textbf{C*-algebraic Smale Mean Value  Conjecture : Let $\mathcal{A}$ be a  commutative  C*-algebra.  	Let $P(z)= (z-a_1)\cdots (z-a_n)$ be a polynomial of degree $n\geq 2$ over $\mathcal{A}$, $a_1, \dots, a_n \in \mathcal{A}$. If $z\in\mathcal{A}$ is not a critical point of $P$, then there exists a critical point $w\in \mathcal{A}$ of $P$ such that 
\begin{align*}
	\frac{\|P(z)-P(w)\|}{\|z-w\|}\leq 1 \|P'(z)\|
\end{align*}
or 
\begin{align*}
	\frac{\|P(z)-P(w)\|}{\|z-w\|}\leq \frac{n-1}{n} \|P'(z)\|=\frac{\operatorname{deg} (P)-1}{\operatorname{deg} (P)} \|P'(z)\|.	
	\end{align*}}
 \textbf{C*-algebraic Dubinin-Sugawa Dual  Mean Value  Conjecture :  Let $\mathcal{A}$ be a commutative    C*-algebra.  	Let $P(z)= (z-a_1)\cdots (z-a_n)$ be a polynomial of degree $n\geq 2$ over $\mathcal{A}$, $a_1, \dots, a_n \in \mathcal{A}$. If $z\in \mathcal{A}$ is not a critical point of $P$, then there exists a critical point $w\in \mathcal{A}$ of $P$ such that 
	\begin{align*}
	\frac{\|P'(z)\|}{\operatorname{deg} (P)}	=\frac{\|P'(z)\|}{n}	\leq  \frac{\|P(z)-P(w)\|}{\|z-w\|}.
\end{align*}}	
We show that (even a strong form of) C*-algebraic Smale mean value  conjecture and C*-algebraic Dubinin-Sugawa dual  mean value  conjecture hold for  degree 2 C*-algebraic polynomials over commutative C*-algebras.

\hrule
\tableofcontents
\hrule
\section{Introduction}
In 1981 Fields Medalist Prof. Steve Smale proved the following result using Koebe 1/4-theorem \cite{SMALE1981}.
\begin{theorem}\cite{SMALE1981} \textbf{(Smale Mean Value Theorem) \label{SMALEMEANTHEOREM}
	Let $P(z)\coloneqq a_0+a_1z+a_2z^2+\cdots+a_nz^n$ be a polynomial of degree $n\geq 2$ over $\mathbb{C}$. If $z\in\mathbb{C}$ is not a critical point of $P$ (i.e., $P'(z)\neq 0$), then 
	\begin{align}\label{SMALEIN}
		\min _{w\in \mathbb{C}, P'(w)=0}\left|\frac{P(z)-P(w)}{z-w}\right|\leq 4 |P'(z)|.
	\end{align}}
\end{theorem}
Note that Theorem \ref{SMALEMEANTHEOREM} can also be stated as follows. 
\begin{theorem} \cite{SMALE1981} \textbf{(Smale Mean Value Theorem)\label{SMALETHEOREM}
Let $P(z)\coloneqq a_0+a_1z+a_2z^2+\cdots+a_nz^n$ be a polynomial of degree $n\geq 2$ over $\mathbb{C}$. If $z\in\mathbb{C}$ is not a critical point of $P$, then there exists a critical point $w\in \mathbb{C}$ of $P$ such that 
\begin{align}\label{CRIT}
	\left|\frac{P(z)-P(w)}{z-w}\right|\leq 4 |P'(z)|.
\end{align}	}
\end{theorem}
Proof of Theorem \ref{SMALETHEOREM} can also be found in book \cite{RAHMANSCHMEISSERBOOK} and  in the monograph \cite{SHEILBOOK}. Inequality (\ref{SMALEIN}) lead Smale to ask the following problem. What is the best constant $4>c>0$ such that the Inequality (\ref{SMALEIN})  can be replaced by 
\begin{align}\label{SMALEC}
	\min _{w\in \mathbb{C}, P'(w)=0}\left|\frac{P(z)-P(w)}{z-w}\right|\leq c |P'(z)|?	
\end{align}
Smale made the following conjecture with regard to Inequality (\ref{SMALEC}) which is open today.
\begin{conjecture} \cite{SMALEIN, SMALE1981, SMALE2000, SHUBSMALE2, SMALE1985, SHUBSMALE1}\label{SMALECONJECTURE} \textbf{(Smale Mean Value Conjecture)
	Let $P(z)\coloneqq a_0+a_1z+a_2z^2+\cdots+a_nz^n$ be a polynomial of degree $n\geq 2$ over $\mathbb{C}$. If $z\in\mathbb{C}$ is not a critical point of $P$, then 
\begin{align*}
	\min _{w\in \mathbb{C}, P'(w)=0}\left|\frac{P(z)-P(w)}{z-w}\right|\leq 1 |P'(z)|
\end{align*}	
or 
\begin{align*}
	\min _{w\in \mathbb{C}, P'(w)=0}\left|\frac{P(z)-P(w)}{z-w}\right|\leq \frac{n-1}{n} |P'(z)|=\frac{\operatorname{deg} (P)-1}{\operatorname{deg} (P)} |P'(z)|.
\end{align*}}
\end{conjecture}
Conjecture  \ref{SMALECONJECTURE} also appears as the Problem 1 in the Addenda after  list of 18 open problems (of which 3 are completely solved now, namely \textbf{Problem  2: Poincare Conjecture} by Perelman in 2003 \cite{MORGANTIAN}, \textbf{Problem 14: Lorenz Attractor} by Tucker in 2002 \cite{TUCKER}, \textbf{Problem 17: Solving Polynomial Equations} by  Beltron and Pardo in 2009 \cite{BELTRANPARDO, BELTRANPARDO2}) of Smale in \textbf{Mathematical Problems for the  Next Century} \cite{SMALE2000}. Now for a polynomial $P$  of degree $n\geq2$, define 
\begin{align*}
	S(P)\coloneqq \sup_{z\in \mathbb{C},P'(z)\neq 0}\frac{1}{|P'(z)|}	\min _{w\in \mathbb{C}, P'(w)=0}\left|\frac{P(z)-P(w)}{z-w}\right|
\end{align*}
and 
\begin{align*}
	M\coloneqq \sup_{P \text{ is a non linear polynomial}} S(P).
\end{align*}
 Smale conjecture then asks whether $M=1$  or $S(P)=\frac{\text{deg}(P)-1}{\text{deg}(P)}$? It is known that the Conjecture \ref{SMALECONJECTURE} holds for polynomials of degree 2, 3, 4, \cite{SHEILBOOK, TISCHLER, RAHMANSCHMEISSERBOOK}. Through computer calculations, Smale conjecture is known to hold for polynomials of degree upto 10 \cite{SENDOVMARINOV}. 

Usually  the study of Smale mean value conjecture goes with the following   conjecture which is known as  normalization equivalence to Conjecture  \ref{SMALECONJECTURE}.
\begin{conjecture}\cite{SMALE1981,SHEILBOOK, RAHMANSCHMEISSERBOOK}\label{SMALENORMAL}
	 \textbf{(Normalized Smale Mean Value Conjecture) Let $P$ be a degree $n\geq2$ polynomial  over $\mathbb{C}$ which is normalized, i.e.,  $P(0)=0$ and $P'(0)=1$. Then 
	\begin{align*}
		S_0(P)\coloneqq 	\min _{w\in \mathbb{C}, P'(w)=0}\left|\frac{P(w)}{w}\right|=\frac{\operatorname{deg}(P)-1}{\operatorname{deg}(P)}
	\end{align*}
	or 
	\begin{align*}
		M_0\coloneqq \sup_{P \operatorname{ is \, a \, non \, linear \, polynomial}, P(0)=0, P'(0)=1} S_0(P)=1.	
	\end{align*}}
\end{conjecture}
In 1996 Andrievskii  and Ruscheweyh  showed that  if  critical points of a normalized polynomial $P$ lie on the unit circle, then $S_0(P)\leq 1$ \cite{ANDRRUS}.  In 2009 Hinkkanen and  Kayumov showed that if all the critical points of a normalized polynomial $P$ are real, then $S_0(P)\leq \frac{2}{3}$ \cite{HINKKANENKAYUMOV2}. 
  In 2010 Hinkkanen  and Kayumov proved that if the critical points of normalized polynomial $P$ lie on the sector $\{re^{i\theta}:r>0,|\theta|\leq \pi/6\}$, then $S_0(P)\leq \frac{1}{2}$   and if the critical points of normalized polynomial  $P$ lie in the union of two rays $\{1+re^{\pm i\theta}: r\geq 0\}$ for each $0<\theta\leq \pi/2$, then $S_0(P)\leq \frac{2}{3}$  \cite{HINKKANENKAYUMOV3}. In \cite{RAHMANSCHMEISSERBOOK} it is showed that if the critical points of a normalized polynomial are all real, then $S_0(P)\leq \frac{\text{deg}(P)-2}{\text{deg}(P)}\left[\left(\frac{\text{deg}(P)-1}{\text{deg}(P)-2}\right)^{{\text{deg}(P)-1}}-2\right]$ and in  \cite{SHEILBOOK} it is showed that $S_0(P)\leq e-2$. \\
   In  2002 Beardon, Minda and Ng showed that $S(P)\leq  4^\frac{\text{deg}(P)-2}{\text{deg}(P)-1}$ \cite{BEARDONMINDANG}.  In 2006 Fijikawa and Sugawa proved that $S(P)\leq  4 \frac{1+(\text{deg}(P)-2)4^\frac{-1}{\text{deg}(P)-1}}{\text{deg}(P)+1}$ \cite{FUJIKAWASUGAWA}.  In 2007 Conte, Fujikawa and Lakic showed that $S(P)\leq  4 \frac{\text{deg}(P)-1}{\text{deg}(P)+1}$ \cite{CONTEFUJIKAWALAKIC}. In 2007 Crane showed that there exists $n_0\geq2$ such that for all $n\geq n_0$, $S(P)\leq 4-\frac{2}{\sqrt{\text{deg}(P)}}$ \cite{CRANE} (also see \cite{CRANE2006}). In 2003 Ng proved that Conjecture \ref{SMALECONJECTURE} holds  with $M=2$ for   all odd polynomials with nonzero linear terms \cite{NG}. In 2006 Dubinin proved that if the critical points of a normalized polynomial $P$ have equal modulus, then $S_0(P)\leq \frac{\operatorname{deg}(P)-1}{\operatorname{deg}(P)}$ \cite{DUBININ}.

At the other side, in 2002  Sendov formulated a generalization of Conjecture \ref{SMALECONJECTURE}  \cite{SENDOV2002}.  In 2003 Sendov and Nikolov gave some variations of Conjecture \ref{SMALECONJECTURE} \cite{SENDOVNIKOLOV}. In 2005 Tyson showed  that Tischler's strong form of Smale mean value conjecture \cite{TISCHLER, TISCHLER1994}  fails \cite{TYSON}.  Further, it is important to record that one can  interpret Conjecture \ref{SMALECONJECTURE} in the form of finite Blaschke products \cite{NGTSANGCHAPETR, NGZHANG, SHEILBOOK}. It is also possible to connect Conjecture \ref{SMALECONJECTURE} with electrostatics \cite{DIMITROV}. \\
In  \cite{SMALE1981} Smale also derived the following higher order mean value theorem. 
\begin{theorem} \cite{SMALE1981} \label{HIGHER}
\textbf{(Higher Order Smale Mean Value Theorem) Let $P(z)\coloneqq a_0+a_1z+a_2z^2+\cdots+a_nz^n$ be a polynomial of degree $n\geq 2$ over $\mathbb{C}$. If $z\in\mathbb{C}$ is not a critical point of $P$, then there exists a critical point $w\in \mathbb{C}$ of $P$ such that 
	\begin{align}\label{CRITHIGHER}
	\frac{|P^{(k)}(z)|}{k!}	\frac{|P(z)-P(w)|^{k-1}}{|P'(z)|^k}\leq 4^{k-1}, \quad \forall 2\leq k \leq n.
\end{align}}	
\end{theorem}
A conjecture regarding the convergence of iterates of critical point satisfying Inequality (\ref{CRIT}) has been made by Miles-Leighton and Pilgrim in 2012 \cite{MILESLEIGHTONPILGRIM}.
\begin{conjecture}\cite{MILESLEIGHTONPILGRIM} \label{LPCONJECTURE}
\textbf{(Miles-Leighton-Pilgrim Dynamics Conjecture)
Let $P(z)\coloneqq z+a_2z^2+\cdots+a_nz^n$ be a polynomial of degree $n\geq 2$ over $\mathbb{C}$. Then there exists a critical point $w\in \mathbb{C}$ of $P$ such that 
\begin{align*}
	\left|\frac{P(w)}{w}\right|\leq 1
\end{align*}	
and 
\begin{align*}
	P^m(w) \to 0\quad \text{ as } \quad m \to \infty.
\end{align*}}
\end{conjecture}
Miles-Leighton and Pilgrim succeeded in showing that Conjecture \ref{LPCONJECTURE} holds for polynomials of degree 2 and 3. 
On the other side, in 2009  Dubinin and Sugawa derived the following striking result which is a dual result for Smale mean value theorem \cite{DUBININSUGAWA}. 
\begin{theorem} \cite{DUBININSUGAWA, DUBININ2012} \textbf{(Dubinin-Sugawa Dual Mean Value Theorem)
Let $P(z)\coloneqq a_0+a_1z+a_2z^2+\cdots+a_nz^n$ be a polynomial of degree $n\geq 2$ over $\mathbb{C}$. If $z\in\mathbb{C}$ is not a critical point of $P$, then there exists a critical point $w\in \mathbb{C}$ of $P$ such that 
\begin{align*}
\frac{|P'(z)|}{n4^n}	\leq \left|\frac{P(z)-P(w)}{z-w}\right|.
\end{align*}
In other words, 
\begin{align*}
\frac{|P'(z)|}{n4^n}\leq 	\max _{w\in \mathbb{C}, P'(w)=0}\left|\frac{P(z)-P(w)}{z-w}\right|.
\end{align*}}	
\end{theorem}
Similar to Smale conjecture, Dubinin and Sugawa made  the following conjecture in 2009.
\begin{conjecture}\cite{DUBININSUGAWA, DUBININ2012} \label{DUBINSUGAWACONJECTURE} \textbf{(Dubinin-Sugawa Dual Mean Value Conjecture)
	Let $P(z)\coloneqq a_0+a_1z+a_2z^2+\cdots+a_nz^n$ be a polynomial of degree $n\geq 2$ over $\mathbb{C}$. If $z\in\mathbb{C}$ is not a critical point of $P$, then 
\begin{align*}
\frac{|P'(z)|}{n}	\leq \max _{w\in \mathbb{C}, P'(w)=0}\left|\frac{P(z)-P(w)}{z-w}\right|.
\end{align*}}	
\end{conjecture}
Similar to earlier quantities, for a polynomial $P$  of degree $n\geq2$, define 
\begin{align*}
	DS(P)\coloneqq \inf_{z\in \mathbb{C},P'(z)\neq 0}\frac{1}{|P'(z)|}	\max _{w\in \mathbb{C}, P'(w)=0}\left|\frac{P(z)-P(w)}{z-w}\right|.
\end{align*}
Similar to Conjecture \ref{SMALENORMAL}, we have the following normalized version of  Conjecture \ref{DUBINSUGAWACONJECTURE}.
\begin{conjecture}\cite{DUBININSUGAWA, DUBININ2012}
 \textbf{(Normalized Smale Dubinin-Sugawa Dual Mean Value Conjecture) Let $P$ be a degree $n\geq2$ polynomial  over $\mathbb{C}$ which is normalized, i.e.,  $P(0)=0$ and $P'(0)=1$. Then 
\begin{align*}
	DS_0(P)\coloneqq 	\max _{w\in \mathbb{C}, P'(w)=0}\left|\frac{P(w)}{w}\right|=\frac{1}{n}=\frac{1}{\operatorname{deg}(P)}.
\end{align*}}	
\end{conjecture}
In 2016 Ng and Zhang showed that  if $P$ is a normalized polynomial, then $	DS_0(P)> \frac{1}{4^\text{deg (P)}}$ \cite{NGZHANG}. In 2019 Dubinin showed  that $	DS(P)\geq \frac{1}{n} \tan \left(\frac{\pi}{4 \text{deg (P)}}\right)$ \cite{DUBININ2019}. 
In 2019  Hinkkanen,  Kayumov  and Khammatova showed that Conjecture \ref{DUBINSUGAWACONJECTURE} holds for polynomials of degree six \cite{HINKKANENKAYUMOVKHAMMATOVA}. In this paper we formulate C*-algebraic versions of Conjectures \ref{SMALECONJECTURE} and \ref{DUBINSUGAWACONJECTURE}. We show that they are valid for degree 2 C*-algebraic polynomials. We also formulate C*-algebraic version of Conjecture  \ref{HIGHER} and Conjecture \ref{LPCONJECTURE}.

  \section{C*-algebraic Smale Mean Value  Conjecture}
  	Let $\mathcal{A}$ be a   C*-algebra. For  $P(z)	\coloneqq (z-a_1)(z-a_2)\cdots (z-a_n)$ for all  $z\in \mathcal{A}$ with  $a_1, a_2, \dots, a_n \in \mathcal{A} $, we define 
  \begin{align*}
  	P'(z)=\sum_{j=1}^{n}(z-a_1)\cdots \widehat{(z-a_j)}\cdots (z-a_n), \quad \forall z \in \mathcal{A}
  \end{align*}
  where the term with cap is missing. We state C*-algebraic version of Conjecture \ref{SMALECONJECTURE} as follows.
  \begin{conjecture} \textbf{(C*-algebraic Smale Mean Value  Conjecture)\label{CSMALE}
 Let $\mathcal{A}$ be a  commutative  C*-algebra.  	Let $P(z)\coloneqq (z-a_1)\cdots (z-a_n)$ be a polynomial of degree $n\geq 2$ over $\mathcal{A}$, $a_1, \dots, a_n \in \mathcal{A}$. If $z\in\mathcal{A}$ is not a critical point of $P$ (i.e., $P'(z)\neq0$), then there exists a critical point $w\in \mathcal{A}$ of $P$ such that 
  	\begin{align*}
  		\frac{\|P(z)-P(w)\|}{\|z-w\|}\leq 1 \|P'(z)\|
  \end{align*}
or 
\begin{align*}
	\frac{\|P(z)-P(w)\|}{\|z-w\|}\leq \frac{n-1}{n} \|P'(z)\|=\frac{\operatorname{deg} (P)-1}{\operatorname{deg} (P)} \|P'(z)\|.
\end{align*}}
  \end{conjecture}
We can formulate a stronger conjecture than  	Conjecture  \ref{CSMALE} as follows. 
 \begin{conjecture} \textbf{(C*-algebraic Smale Mean Value  Conjecture - strong form)}\label{STRONGERCSMALE}
\textbf{Let $\mathcal{A}$ be a commutative  C*-algebra.  	Let $P(z)\coloneqq (z-a_1)\cdots (z-a_n)$ be a polynomial of degree $n\geq 2$ over $\mathcal{A}$, $a_1, \dots, a_n \in\mathcal{A}$. If $z\in\mathcal{A}$ is not a critical point of $P$, then there exists a critical point $w\in \mathcal{A}$ of $P$ such that 
\begin{align}\label{ST}
	(P(z)-P(w))(P(z)-P(w))^*\leq  (z-w)P'(z)P'(z)^*(z-w)^*
\end{align}
or 
\begin{align}\label{SST}
		(P(z)-P(w))(P(z)-P(w))^*&\leq \left(\frac{n-1}{n}\right)^2(z-w)P'(z)P'(z)^*(z-w)^*\\
		&=\left(\frac{\operatorname{deg} (P)-1}{\operatorname{deg} (P)} \right)^2(z-w)P'(z)P'(z)^*(z-w)^*.\nonumber
\end{align}}
\end{conjecture}
\begin{proposition}
	Conjecture  \ref{STRONGERCSMALE}  implies Conjecture \ref{CSMALE}.
\end{proposition}
\begin{proof}
	Taking norm in Inequality (\ref{ST}) we get 
	\begin{align*}
	\|P(z)-P(w)\|^2&=	\|(P(z)-P(w))(P(z)-P(w))^*\|\leq  \|(z-w)P'(z)P'(z)^*(z-w)^*\|\\
	&= \|(z-w)P'(z)\|^2\leq  \|z-w\|^2\|P'(z)\|^2.
	\end{align*}
Similarly by taking norm in Inequality (\ref{SST}) we get the second inequality in Conjecture  \ref{CSMALE}.
\end{proof}
  \begin{theorem}\label{CH}
  	Conjecture  \ref{CSMALE}  holds for degree 2 C*-algebraic polynomials over   C*-algebras. 
  \end{theorem}
  \begin{proof}
  	We prove that Conjecture  \ref{STRONGERCSMALE}  holds and hence 	Conjecture  \ref{CSMALE} holds. Let $\mathcal{A}$ be a commutative C*-algebra and let $P(z)\coloneqq (z-a)(z-b)$, $a, b \in \mathcal{A}$. Then $P'(z)=2z-(a+b)$ for all $z\in \mathcal{A}$. Let $z\in \mathcal{A}$ be such that $P'(z)\neq0$. Note that $P'$ has only one zero, namely $c\coloneqq\frac{a+b}{2}$. To show Inequality (\ref{SST}), we need to show that 
  	\begin{align*}
  		(P(z)-P(c))(P(z)-P(c))^*\leq \frac{1}{4} (z-c)P'(z)P'(z)^*(z-c)^*.	
  	\end{align*}
Define $x\coloneqq z-a$ and $y\coloneqq z-b$.   Then
  \begin{align*}
  	&z-c=z-\frac{a+b}{2}=\frac{z-a+z-b}{2}=\frac{x+y}{2},\\
  	&P'(z)=z-a+z-b=x+y,\\
  	&P(c)=\left(\frac{a+b}{2}-a\right)\left(\frac{a+b}{2}-b\right)=\frac{-(b-a)^2}{4}=\frac{-(x-y)^2}{4}, 
  \end{align*}
and 
\begin{align*}
(P(z)-P(c))(P(z)-P(c))^*=\left(xy+\frac{(x-y)^2}{4}\right)\left(xy+\frac{(x-y)^2}{4}\right)^*=\frac{(x+y)^2(x^*+y^*)^2}{16}.
\end{align*}
Therefore
\begin{align*}
	&\frac{1}{4} (z-c)P'(z)P'(z)^*(z-c)^*-(P(z)-P(c))(P(z)-P(c))^*\\
	&=\frac{1}{4}\left(\frac{x+y}{2}\right)(x+y)(x^*+y^*)\left(\frac{x^*+y^*}{2}\right)-\frac{1}{16}(x+y)^2(x^*+y^*)^2\\
	&=\frac{1}{16}((x+y)^2(x^*+y^*)^2-(x+y)^2(x^*+y^*)^2)=0.
\end{align*}
  \end{proof}
Based on Theorem \ref{HIGHER} we formulate the following higher order conjecture. 
\begin{conjecture}\label{HIGHERMEAN}
\textbf{(Higher Order C*-algebraic Smale Mean Value  Conjecture)
	Let $\mathcal{A}$ be a commutative  C*-algebra.	Let $P(z)\coloneqq (z-a_1)\cdots (z-a_n)$ be a polynomial of degree $n\geq 2$ over $\mathcal{A}$, $a_1, \dots, a_n \in \mathcal{A}$. If $z\in\mathcal{A}$ is not a critical point of $P$, then there exists a critical point $w\in \mathcal{A}$ of $P$ such that 
\begin{align*}
	\frac{\|P^{(k)}(z)\|}{k!}	\frac{\|P(z)-P(w)\|^{k-1}}{\|P'(z)\|^k}\leq 4^{k-1}, \quad \forall 2\leq k \leq n.	
\end{align*}}
\end{conjecture}
  \begin{theorem}
	Conjecture \ref{HIGHERMEAN} holds for degree 2 C*-algebraic polynomials. 
\end{theorem}
\begin{proof}
We continue from the proof of Theorem \ref{CH}. Let $z\in \mathcal{A}$ be not a zero of $P'$. We have $P''(z)=2$. Hence 
\begin{align*}
	&\|P''(z)\|=2, \quad \|P'(z)\|=\|x+y\|,\\
	\|P(z)-P(c)\|^2&=\|(P(z)-P(c))(P(z)-P(c))^*\|=\frac{1}{16}\|(x+y)^2(x^*+y^*)^2\|\\
	& =\frac{1}{16}\|(x+y)^2((x+y)^2)^*\| =\frac{1}{16}\|(x+y)^2\|^2\leq \frac{1}{16}\|x+y\|^4.
\end{align*}
Therefore 
\begin{align*}
	\frac{\|P^{(2)}(z)\|}{2!}	\frac{\|P(z)-P(c)\|^{2-1}}{\|P'(z)\|^2}=\frac{2}{2}	\frac{\left(\frac{\|(x+y)^2\|}{4}\right)}{\|x+y\|^2}\leq \frac{1}{4}\frac{\|x+y\|^2}{\|x+y\|^2}=\frac{1}{4}=4^{2-1}.
\end{align*}
\end{proof}   
We state Conjecture \ref{LPCONJECTURE} for C*-algebraic polynomials as follows. 
\begin{conjecture}\label{DYNAMIC}
\textbf{(C*-algebraic Miles-Leighton-Pilgrim Dynamics Conjecture) 	
	Let $\mathcal{A}$ be a  commutative  C*-algebra.  	Let $P(z)\coloneqq (z-a_1)\cdots (z-a_n)$ be a polynomial of degree $n\geq 2$ over $\mathcal{A}$, $a_1, \dots, a_n \in \mathcal{A}$. If $P(0)=0$ and $P'(0)=1$, then   there exists a critical point $w\in \mathcal{A}$ of $P$ such that 
	\begin{align*}
		\frac{\|P(w)\|}{\|w\|}\leq 1
	\end{align*}	
	and 
	\begin{align*}
		P^m(w) \to 0\quad \text{ as } \quad m \to \infty.
\end{align*}}
\end{conjecture}

 \section{C*-algebraic Dubinin-Sugawa Dual  Mean Value  Conjecture}
 We state C*-algebraic version of Conjecture \ref{DUBINSUGAWACONJECTURE} as follows.
 \begin{conjecture}\textbf{(C*-algebraic Dubinin-Sugawa Dual  Mean Value  Conjecture)}\label{DUALSMALE}
  \textbf{Let $\mathcal{A}$ be a  commutative  C*-algebra.  	Let $P(z)\coloneqq (z-a_1)\cdots (z-a_n)$ be a polynomial of degree $n\geq 2$ over $\mathcal{A}$, $a_1, \dots, a_n \in \mathcal{A}$. If $z\in\mathcal{A}$ is not a critical point of $P$, then there exists a critical point $w\in \mathcal{A}$ of $P$ such that 
 \begin{align*}
 \frac{\|P'(z)\|}{\operatorname{deg} (P)}= \frac{\|P'(z)\|}{n}	\leq  \frac{\|P(z)-P(w)\|}{\|z-w\|}.
\end{align*}}	
 \end{conjecture}
Similar to the strong form of Conjecture \ref{CSMALE} we have the following strong form of Conjecture \ref{DUALSMALE}.
 \begin{conjecture}\textbf{(C*-algebraic Dubinin-Sugawa Dual  Mean Value  Conjecture - strong form})\label{STRONGERDUALSMALE}
\textbf{Let $\mathcal{A}$ be a  commutative  C*-algebra.  	Let $P(z)\coloneqq (z-a_1)\cdots (z-a_n)$ be a polynomial of degree $n\geq 2$ over $\mathcal{A}$, $a_1, \dots, a_n \in \mathcal{A}$. If $z\in\mathcal{A}$ is not a critical point of $P$, then there exists a critical point $w\in \mathcal{A}$ of $P$ such that 
	\begin{align*}
	\frac{(z-w)P'(z)P'(z)^*(z-w)^*}{(\operatorname{deg} (P))^2}=\frac{(z-w)P'(z)P'(z)^*(z-w)^*}{n^2}	\leq (P(z)-P(w))(P(z)-P(w))^*.
\end{align*}}	
\end{conjecture}
 \begin{theorem}
 	Conjecture  \ref{DUALSMALE}  holds for degree 2 C*-algebraic polynomials. 
 \end{theorem}
 \begin{proof}
 	We prove Conjecture \ref{STRONGERDUALSMALE} holds. We continue from the proof of Theorem \ref{CH}. Consider 
 	\begin{align*}
 	&(P(z)-P(c))(P(z)-P(c))^*-	\frac{(z-c)P'(z)P'(z)^*(z-c)^*}{2^2}\\
 	&=\frac{(x+y)^2(x^*+y^*)^2}{16}-\frac{1}{4}\frac{(x+y)^2(x^*+y^*)^2}{4}=0.
 	\end{align*}
 \end{proof}
 
\begin{remark}
  	\begin{enumerate}[\upshape(i)]
  	\item Conjectures \ref{CSMALE}, \ref{HIGHERMEAN}, \ref{DYNAMIC} and \ref{DUALSMALE} can be stated for   Banach algebras.
  		\item \textbf{C*-algebraic Sendov conjecture} has been stated in \cite{MAHESHKRISHNA}.
  		\item \textbf{C*-algebraic Schoenberg conjecture} has been stated in \cite{MAHESHKRISHNA2}.
  	\end{enumerate}
  \end{remark}
 \bibliographystyle{plain}
 \bibliography{reference.bib}

\end{document}